\documentclass[proceedings,submission%
]{dmtcs}

\usepackage[latin1]{inputenc}
\usepackage{subfigure}

\usepackage[round]{natbib} 
 
\usepackage{amsmath,amssymb,amsfonts,latexsym, array, float}
\usepackage{tikz}
\usetikzlibrary{shadows}
\usepackage{enumerate}

\newtheorem{defi}{Definition}[section]
\newtheorem{thm}[defi]{Theorem}
\newtheorem{conj}[defi]{Conjecture}

\newtheorem{prop}[defi]{Proposition}
\newtheorem{example}[defi]{Example}

\usepackage{xspace}

\newcommand{\s}{\ensuremath{\mathbf{S}}\xspace}
\newcommand{\sos}{\ensuremath{\mathbf{S} \circ \mathbf{S}}\xspace}
\newcommand{\srs}{\ensuremath{\mathbf{S} \circ \mathbf{r} \circ \mathbf{S}}\xspace}
\newcommand{\sis}{\ensuremath{\mathbf{S} \circ \mathbf{i} \circ \mathbf{S}}\xspace}
\newcommand{\scs}{\ensuremath{\mathbf{S} \circ \mathbf{c} \circ \mathbf{S}}\xspace}

\newcommand{\Id}{\mathrm{Id}}

\newcommand{\sym}{\mathfrak{S}}
\newcommand{\dd}{\text{-}}

\newcommand{\R}{\ensuremath{\mathcal{R}}\xspace}
\newcommand{\C}{\ensuremath{\mathcal{C}}\xspace}

\DeclareMathOperator{\indmax}{\mathrm{indmax}}
\DeclareMathOperator{\comp}{\mathrm{comp}}
\DeclareMathOperator{\lmax}{\mathrm{lmax}}
\DeclareMathOperator{\rmax}{\mathrm{rmax}}
\DeclareMathOperator{\lmin}{\mathrm{lmin}}
\DeclareMathOperator{\rmin}{\mathrm{rmin}}
\DeclareMathOperator{\ldr}{\mathrm{ldr}}
\DeclareMathOperator{\lir}{\mathrm{lir}}
\DeclareMathOperator{\rdr}{\mathrm{rdr}}
\DeclareMathOperator{\rir}{\mathrm{rir}}
\DeclareMathOperator{\zeil}{\mathrm{zeil}}

\DeclareMathOperator{\peak}{\mathrm{peak}}
\DeclareMathOperator{\valley}{\mathrm{valley}}
\DeclareMathOperator{\ddes}{\mathrm{ddes}}
\DeclareMathOperator{\dasc}{\mathrm{dasc}}
\DeclareMathOperator{\des}{\mathrm{des}} 
\DeclareMathOperator{\maj}{\mathrm{maj}} 
\DeclareMathOperator{\slmax}{\mathrm{slmax}}

\begin{document}

\title{Enumeration of permutations sorted with two passes through a stack and $D_8$ symmetries}

\author[M. Bouvel \and O. Guibert]{Mathilde Bouvel\addressmark{1,2} \and Olivier Guibert \addressmark{1}}

\address{\addressmark{1}LaBRI UMR 5800, Univ. Bordeaux and CNRS, Talence, France. \\ \addressmark{2}partially supported by ANR grants PSYCO (2011-JCJC) and MAGNUM (2010-BLAN-0204)}


\maketitle
\begin{abstract}
\paragraph{Abstract.}
We examine the sets of permutations that are sorted by two passes through a stack with a $D_8$ operation performed in between. From a characterization of these in terms of generalized excluded patterns, we prove two conjectures on their enumeration, that can be refined with the distribution of some statistics. The results are obtained by generating trees.

\paragraph{R\'esum\'e.}
On \'etudie les ensembles de permutations qui sont tri\'ees par deux passages dans une pile s\'epar\'es par une op\'eration du groupe $D_8$.  \`A partir d'une caract\'erisation de ces ensembles en termes de motifs exclus g\'en\'eralis\'es, on d\'emontre deux conjectures sur leur \'enum\'eration, qui peuvent \^etre raffin\'ees par la distribution de certaines statistiques. Ces r\'esultats sont obtenus \`a l'aide d'arbres de g\'en\'eration.
\end{abstract}

\section{Introduction}

\subsection{Permutations, diagrams and patterns}

A permutation of $\sym_n$ is a bijective map from $\{1, 2, \ldots, n \}$ to itself, $n$ being called the \emph{size} of the permutation. In our context, we will view permutations in two different ways. A permutation $\sigma$ of $\sym_n$ can be seen as a word $\sigma_1 \sigma_2 \ldots \sigma_n$ where $\sigma_i = \sigma(i)$ for all $i \in \{1, \ldots,n\}$, containing exactly once each letter from $1$ to $n$. It can also be seen as what we call its \emph{diagram}: an $n \times n$ grid with exactly one dot per row and per column, the dots being placed in cells of coordinates $(i, \sigma(i))$. 
For every element of a permutation $\sigma$, corresponding to the dot at coordinates $(i, \sigma(i))$ in its diagram, we call $i$ its \emph{index} and $\sigma(i)$ its \emph{value}. 

Recall that for $\sigma \in \sym_n$, its \emph{reverse} (resp. \emph{complement}, resp. \emph{inverse}) is the permutation $\mathbf{r}(\sigma)$ (resp. $\mathbf{c}(\sigma)$, resp. $\mathbf{i}(\sigma)$) defined by $\mathbf{r}(\sigma)(i) = \sigma(n+1-i)$  (resp. $\mathbf{c}(\sigma)(i) = n+1-\sigma(i)$, resp. $\mathbf{i}(\sigma)(i) = j$ such that $\sigma(j)=i$). These operations correspond respectively to symmetries w.r.t. a vertical axis, a horizontal axis and the south-west to north-east diagonal on the diagrams of the permutations (see Figure \ref{fig:sym}). Hence, these three operations generate the eight element group $D_8$ of the symmetries of the square. 

\begin{figure}[ht]
\begin{center}
\begin{tabular}{ccccccc}
\begin{tikzpicture}
\begin{scope}[scale=0.35]
\draw [help lines] (0,0) grid (6,6);
\node at (-1,-1) {~};
\node at (7,7) {~};
\draw (0.5,3.5) [fill] circle (.2);
\draw (1.5,0.5) [fill] circle (.2);
\draw (2.5,5.5) [fill] circle (.2);
\draw (3.5,1.5) [fill] circle (.2);
\draw (4.5,4.5) [fill] circle (.2);
\draw (5.5,2.5) [fill] circle (.2);
\end{scope}
\end{tikzpicture} & \qquad & 
\begin{tikzpicture}
\begin{scope}[scale=0.35]
\draw [help lines] (0,0) grid (6,6);
\node at (-1,-1) {~};
\node at (7,7) {~};
\draw (0.5,2.5) [fill] circle (.2);
\draw (1.5,4.5) [fill] circle (.2);
\draw (2.5,1.5) [fill] circle (.2);
\draw (3.5,5.5) [fill] circle (.2);
\draw (4.5,0.5) [fill] circle (.2);
\draw (5.5,3.5) [fill] circle (.2);
\draw[very thick, red] (3,-0.5) -- (3,6.5);
\end{scope}
\end{tikzpicture} & \qquad & 
\begin{tikzpicture}
\begin{scope}[scale=0.35]
\draw [help lines] (0,0) grid (6,6);
\node at (-1,-1) {~};
\node at (7,7) {~};
\draw (0.5,2.5) [fill] circle (.2);
\draw (1.5,5.5) [fill] circle (.2);
\draw (2.5,0.5) [fill] circle (.2);
\draw (3.5,4.5) [fill] circle (.2);
\draw (4.5,1.5) [fill] circle (.2);
\draw (5.5,3.5) [fill] circle (.2);
\draw[very thick, red] (-0.5,3) -- (6.5,3);
\end{scope}
\end{tikzpicture} & \qquad & 
\begin{tikzpicture}
\begin{scope}[scale=0.35]
\draw [help lines] (0,0) grid (6,6);
\node at (-1,-1) {~};
\node at (7,7) {~};
\draw (0.5,1.5) [fill] circle (.2);
\draw (1.5,3.5) [fill] circle (.2);
\draw (2.5,5.5) [fill] circle (.2);
\draw (3.5,0.5) [fill] circle (.2);
\draw (5.5,2.5) [fill] circle (.2);
\draw[very thick, red] (-0.3,-0.3) -- (6.3,6.3);
\draw (4.5,4.5) [fill] circle (.2);
\end{scope}
\end{tikzpicture} \\
Original permutation & & Reverse & & Complement & & Inverse \\
\end{tabular}
\caption{The diagram of permutation $\sigma = 416253 $, and its symmetries under $\mathbf{r}$, $\mathbf{c}$ and $\mathbf{i}$.} \label{fig:sym}
\end{center}
\end{figure}
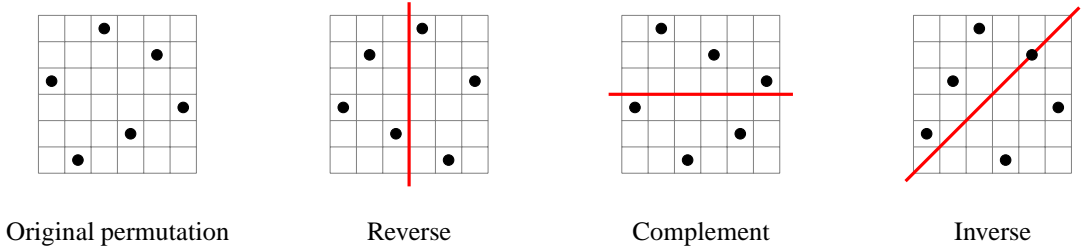

From the word representation of permutations, we inherit basic concepts like word concatenation, or subwords. A subword (with $k$ letters) of a permutation is however not a permutation in general, as its letters may not consist of all integers of $\{1, 2, \ldots, k\}$. Each subword of a permutation $\sigma$ can be \emph{normalized} to a permutation, that is then called a (classical) \emph{pattern} of $\sigma$. More precisely, a permutation $\pi = \pi_1 \pi_2 \ldots \pi_k$ is a (classical) \emph{pattern} of a permutation $\sigma = \sigma_1 \sigma_2 \ldots \sigma_n$ if and only if there exist integers $1 \leq i_1 < i_2 < \ldots < i_k \leq n$ such that $\sigma_{i_1}\ldots \sigma_{i_k}$ is order-isomorphic to $\pi$, \emph{i.e.} such that $\sigma_{i_{\ell}} < \sigma_{i_m}$ whenever $\pi_{\ell} < \pi_{m}$. The subsequence $\sigma_{i_1}\ldots \sigma_{i_k}$ is called an \emph{occurrence} of $\pi$ in $\sigma$. A permutation $\sigma$ that does not contain $\pi$ as a pattern is said to {\em avoid} $\pi$.

Some generalizations of permutation patterns are frequent in the literature, and we will make use of two of them here. 
When introducing \emph{dashes} in a pattern $\pi$, we impose adjacency constraints on the elements $\sigma_{i_{\ell}}$ of $\sigma$ that form an occurrence of $\pi$. Namely, a permutation $\pi = \pi_1 \pi_2 \ldots \pi_k$ with dashes between some pairs $\pi_{\ell} \pi_{\ell+1}$ is a \emph{dashed pattern} of a permutation $\sigma = \sigma_1 \sigma_2 \ldots \sigma_n$ if and only if there exist integers $1 \leq i_1 < i_2 < \ldots < i_k \leq n$ such that $\sigma_{i_1}\ldots \sigma_{i_k}$ is order-isomorphic to $\pi$ and $i_{\ell +1} = i_{\ell} +1$ whenever there is no dash between  $\pi_{\ell}$ and $\pi_{\ell+1}$ in $\pi$. Notice that classical patterns can be viewed as dashed patterns with dashes between $\pi_{\ell}$ and $\pi_{\ell+1}$ for any $\ell \in \{1, \ldots n-1\}$, and that a dashed pattern with no dash corresponds to a normalized \emph{factor} of the permutation. However, in this article, we take the convention that a pattern written with no dash is always a classical pattern. Pattern avoidance of dashed patterns is defined as in the classical case.

The other generalization of pattern avoidance we consider is the one of patterns with one barred element (or barred patterns for short). Consider a permutation $\pi$ with one barred element, and denote by $\pi^{\sim}$ the normalization of the subword of $\pi$ obtained when deleting the barred element. We say that a permutation $\sigma$ \emph{contains} the \emph{barred pattern} $\pi$ if there exists a (classical) occurrence of $\pi^{\sim}$ in $\sigma$ that cannot be extended into a (classical) occurrence of $\pi$ in $\sigma$. 
Consequently, a permutation $\sigma$ \emph{avoids} the barred pattern $\pi$ if every occurrence of $\pi^{\sim}$ in $\sigma$ can be extended into an occurrence of $\pi$ in $\sigma$.

We denote by $Av(\pi', \pi'', \ldots, \pi''')$ the set of permutations that simultaneously avoid $\pi', \pi'', \ldots$ and $\pi'''$.

\begin{example}
Permutation $\sigma = 316452$ avoids the classical pattern $2413$ but contains the classical pattern $\pi = 2431$, and the subwords $3642$ and $3652$ are its two occurrences in $\sigma$. Furthermore $\sigma$ avoids $24\dd 3\dd 1$ as a dashed pattern, as the elements corresponding to $2$ and $4$ are never at consecutive indices in an occurrence of $\pi$ in $\sigma$. Finally, $\sigma$ avoids the barred pattern $\tau = 3\bar{1}542$ as all the occurrences of $\tau^{\sim} = \pi$ can be extended with a smallest element to account for $\bar{1}$ .
\end{example}

\subsection{Some sorting operators on permutations}

Barred patterns have been introduced for characterizing the permutations that can be sorted by two passes of the stack sorting operator \s in \cite{West93} (see Theorem \ref{thm:characterization} below). We say that a permutation $\sigma \in \sym_n$ is sorted by an operator $\mathbf{Sort}$ when $\mathbf{Sort}(\sigma) = 1 2 \ldots n$.

The operation of sorting a permutation $\sigma = \sigma_1 \ldots \sigma_n$ through a stack is defined as follows. Consider a stack that satisfies the Hanoi condition, \emph{i.e.} such that the elements in the stack are in increasing order from the top to the bottom of the stack. Starting from $w = \sigma_1 \ldots \sigma_n$ and an empty stack, either put the first letter of $w$ on the stack (if this respects the Hanoi condition), or otherwise pop the element at the top of the stack. When $w$ and the stack are empty, a permutation has been output, which is the result of the stack sorting of $\sigma$. More formally, the \emph{stack sorting operator} is also classically characterized recursively by $\s(\varepsilon) = \varepsilon$ and $\s(LnR) = \s(L) \s(R) n$ where $n$ is the maximum of the word $LnR$ of distinct integers (which is not necessarily a permutation) and $\varepsilon$ denotes the empty word. Other sorting operators on permutations have been studied in the literature, in connexion with permutation patterns. This is the case in \cite{ClDuSt} for the \emph{tack sorting operator} $\mathbf{T}$ defined by $\mathbf{T}(\varepsilon) = \varepsilon$ and $\mathbf{T}(LnR) = \mathbf{T}(R) \mathbf{T}(L) n$ or in \cite{AlAtBoClDu11} for the \emph{bubble sort operator} $\mathbf{B}$ defined by $\mathbf{B}(\varepsilon) = \varepsilon$ and $\mathbf{B}(LnR) = \mathbf{B}(L) R n$.

The tack sorting operator can be easily characterized by the identity $\mathbf{T}(\sigma) = \s(\mathbf{r}(\sigma))$ for every permutation $\sigma$. Therefore, the compositions of these two sorting operators can be interpreted as $\s \circ \mathbf{T} = \s \circ \s \circ \mathbf{r}$ and $\mathbf{T} \circ \s = \srs$. Similarly, $\mathbf{T} \circ \mathbf{T} = \srs \circ \mathbf{r}$. Hence, following the line of \cite{West93} and looking for a characterization of the permutations that are sorted by these compositions of sorting operators, we are lead to the analysis of the permutations that are sorted by $\srs$. We actually address a rather more general question here: we characterize and enumerate permutations that are sorted by $\s \circ \alpha \circ \s$ for any $\alpha$ in the group $D_8$.

\subsection{Permutations sorted by the composition of two sorting operators}

For any sorting operator $\mathbf{Sort}$, let us denote by $\Id(\mathbf{Sort})$ the set of permutations that are sorted by $\mathbf{Sort}$, \emph{i.e.} $\Id(\mathbf{Sort}) = \cup_n \{\sigma \in \sym_n: \mathbf{Sort}(\sigma)= 1 2 \ldots n\}$. It has been known since \cite{Knuth:ArtComputerProgramming:1:1973} that $\Id(\s) = Av(2 3 1)$, and \cite{West93} has proved that $\Id(\sos) = Av(2 3 4 1, 3\bar{5} 2 4 1)$. Theorem~\ref{thm:characterization} below has been proved by \cite{thm1.2}.

\begin{thm} The sets of permutations that are sorted by $S \circ  \alpha \circ S$, for any $\alpha$ in $D_8$ are characterized by:
\vspace{-6.5pt}
\begin{enumerate}[(i)]\setlength{\itemsep}{-0.8pt}
 \item $\Id(\sos) = \Id(\mathbf{S}\circ \mathbf{i} \circ \mathbf{c} \circ \mathbf{r} \circ \mathbf{S}) = Av(2 3 4 1, 3\bar{5} 2 4 1)$;
 \item $\Id(\scs) = \Id(\mathbf{S}\circ \mathbf{i} \circ \mathbf{r} \circ \mathbf{S}) = Av(2 3 1)$;
 \item $\Id(\srs) = \Id(\mathbf{S}\circ \mathbf{i} \circ \mathbf{c} \circ \mathbf{S}) = Av(1 3 4 2, 31\dd 4\dd 2) = Av(1 3 4 2, 3 \bar{5} 1 4 2)$;
 \item $\Id(\sis) = \Id(\mathbf{S}\circ \mathbf{r} \circ \mathbf{c} \circ \mathbf{S}) = Av(3 4 1 2, 3\dd 4\dd 21)$.
\end{enumerate}
\vspace{-9pt}
\label{thm:characterization}
\end{thm}

A natural question is then to look at the enumeration sequences $(c_n)$ of the sets $\mathcal{C}$ of pattern avoiding permutations that appear in Theorem~\ref{thm:characterization}. Of course, the set $Av(231)$ (which corresponds to one-stack sortable permutations) is enumerated by the Catalan numbers $Cat_n = \frac{1}{n+1} {{2n}  \choose n}$ (see \cite{Knuth:ArtComputerProgramming:1:1973}); and \cite{DGG98} proved that the set $Av(2 3 4 1, 3\bar{5} 2 4 1)$ of two-stack sortable permutations is enumerated by $\frac{2 (3n)!}{(n+1)! (2n+1)!}$. For the two other sets, conjectures on their enumeration have been proposed by~\cite{ClDuSt}, and refined with the distribution of some statistics. These conjectures are stated as Theorems~\ref{thm:conj6} and~\ref{thm:conj7} and Conjecture~\ref{conj:conj7} below. Theorems~\ref{thm:conj6} and~\ref{thm:conj7} are proved in the rest of this article, and the proof of Conjecture~\ref{conj:conj7} is a work in progress.

\begin{thm}
The two sets $\Id(\sos)$ and $\Id(\srs)$ are enumerated according to the size of the permutations by the same sequence.  Moreover, the $15$-tuple of statistics $( \des,\maj,\rmax,\lmax,\valley,\peak,\ddes,$ $\dasc, \zeil,\indmax,\rir,\rdr,\lir,\ldr,\slmax )$ has the same distribution on both sets.
\label{thm:conj6}
\end{thm}

\begin{thm}
The set $\Id(\sis)$ is enumerated by the Baxter numbers.
\label{thm:conj7}
\end{thm}

\begin{conj}
The triple of statistics $(\des, \lmax, \comp)$ has the same distribution on $\Id(\sis)$ and on the set of Baxter permutations. 
\label{conj:conj7}
\end{conj}

The definitions of the statistics of Theorem~\ref{thm:conj6} and Conjecture~\ref{conj:conj7} are briefly recalled in Table~\ref{table:def_of_statistics}. More detailed definitions can be found in~\cite{ClKi08} for instance.

\begin{table}[ht]
\begin{tabular}{|l|l|}
\hline
$\des$ & Number of descents\\
\hline
$\maj$ & Major index, \emph{i.e.} sum of the indices of the descents\\
\hline
$\comp$ & Number of components \\
& Largest $k$ such that $\pi$ can be written as the concatenation $\pi = \alpha_1 \cdot \ldots \cdot \alpha_k$\\
& with for all $i<j$, for all $a_i \in \alpha_i$ and $a_j \in \alpha_j$, $a_i < a_j$.\\ 
& The segments $\alpha_i$ are called the components of $\pi$.\\
\hline
$\rmax,\rmin, \lmax,\lmin$ & Number of right-to-left (resp. left-to-right) maxima (resp. minima)\\
\hline
$\valley, \peak, \ddes,\dasc$ & Number of valleys (resp. peaks, double descents, double ascents)\\
\hline
$\rir, \rdr, \lir, \ldr$ & Length of the rightmost (resp. leftmost) increasing (resp. decreasing) run\\
\hline
$\indmax$ & Index of the maximal element\\ 
\hline
$\zeil$ & Largest $k$ such that $n (n-1) \ldots (n-k+1)$ is a subword of $\pi$ with $n=|\pi|$\\
\hline
$\slmax$ & Largest $k$ such that $\pi_1 \geq \pi_i$ for all $i \in [1..k]$\\
\hline
\end{tabular}
\smallskip
\caption{Some classical statistics on permutations.}
\label{table:def_of_statistics}
\end{table}

The set $Bax$ of Baxter permutations has been first defined in~\cite{Bax64} and can be characterized by excluded dashed patterns as $Bax = Av(2 \dd 41 \dd 3, 3 \dd 14 \dd 2)$ (see~\cite{Ouc05} for example). In this article, we take this as the definition of Baxter permutations. The Baxter numbers $(b_n)$ enumerate the set of Baxter permutations, and we have $b_n = \frac{2}{n(n+1)^2} \sum_{k=1}^n {{n+1} \choose {k-1}} {{n+1} \choose {k}} {{n+1} \choose {k+1}}$ (see~\cite{CGHK78}).

Theorems~\ref{thm:conj6} and~\ref{thm:conj7} are proved using generating trees, and we recall the guidelines of this method in Section~\ref{sec:gen_trees}. Theorem~\ref{thm:conj6} is then proved in Section~\ref{sec:srs}. Section~\ref{sec:sis_enum} proves the enumerative result of Theorem~\ref{thm:conj7} by establishing a bijective correspondence (\emph{via} a generating tree) between $\Id(\sis)$ and $Av(2\dd 14 \dd 3, 2\dd 41 \dd 3)$, this set being enumerated by the Baxter numbers (see~\cite{Gui95} or more recently~\cite{Gir11}). But this does not allow to follow the three statistics $(\des, \lmax, \comp)$ of Conjecture~\ref{conj:conj7} directly. Going into further details of the bijection described by~\cite{Gir11} and examining how the statistics are transformed by this bijection is a promising path to a proof of Conjecture~\ref{conj:conj7}.

\section{Rewriting systems and generating trees}
\label{sec:gen_trees}

\emph{Generating trees} have been first introduced by \cite{Wes95} in the context of pattern avoiding permutations. \cite{BDPP99} have extended the definition of generating trees to other combinatorial objects, as well as formalized it by means of the \emph{ECO-method}. Here, we only recall some basics on generating trees and rewriting systems, before we apply this method in the proofs of Theorems~\ref{thm:conj6} and~\ref{thm:conj7}.

\subsection{Generating trees}

A {\em generating tree} is an infinite rooted tree associated to a set \C of permutations, whose vertices are permutations. The root is the permutation $1 \in \C$, and permutations at distance $n-1$ from the root are exactly the permutations of size $n$ in \C, whose set is denoted by $\C_n$.

There are several possible definitions of the edges of a generating tree. The four usual ones are as follow: the parent of a permutation $\pi \in \C_n$ is the permutation of $\C_{n-1}$ obtained from $\pi$ by removing the largest (resp. smallest, rightmost, leftmost) element of the permutation, and normalizing. This obviously defines uniquely the parent of a permutation $\pi \neq 1$.
Of course, depending on \C and on the operation of removing and normalizing that is chosen, it may result in a permutation that does not belong to \C. This is never the case when \C is a pattern class, \emph{i.e.} when the excluded patterns are classical patterns, but it may happen when the excluded patterns are generalized patterns. In the case of barred patterns, the situations in which it does not happen have been characterized in \cite[Prop. 6]{DGG98}.

The above describes the way a parent is obtained from any of its children in the generating tree. In order to build the tree efficiently from its root, it is certainly more convenient to describe instead how the children of a permutation $\pi$ are built from $\pi$. There are of course four possible {\em insertion rules} corresponding to the four possibilities of removing an element that have been described earlier. For each insertion rule, we define the {\em sites} of a permutation $\pi \in \C$ to be the places in which a new element may be inserted, resulting in a permutation $\pi'$. The locations of these sites that are described below are intented with respect to the diagram that represents permutation $\pi$. A site is said {\em active} when $\pi' \in \C$.

\vspace*{-0.25cm}
\paragraph*{Insertion rule {\em Largest}}
The sites of a permutation $\pi \in \C_n$ are above $\pi$, and are located at the beginning, at the end and between any two consecutive indices $i$ and $i+1$ of $\pi$. The children of $\pi$ are the permutations of $\C_{n+1}$ obtained when inserting a largest element in an active site of~$\pi$.

\vspace*{-0.25cm}
\paragraph*{Insertion rule {\em Smallest}}
The sites of a permutation $\pi \in \C_n$ are below $\pi$, and are located at the beginning, at the end and between any two consecutive indices $i$ and $i+1$ of $\pi$. The children of $\pi$ are the permutations of $\C_{n+1}$ obtained when inserting a smallest element in an active site of~$\pi$.

\vspace*{-0.25cm}
\paragraph*{Insertion rule {\em Rightmost}}
The sites of a permutation $\pi \in \C_n$ are to the right of $\pi$, and are located below, above and between any two consecutive values $i$ and $i+1$ of $\pi$. The children of $\pi$ are the permutations of $\C_{n+1}$ obtained when inserting a rightmost element in an active site of~$\pi$.

\vspace*{-0.25cm}
\paragraph*{Insertion rule {\em Leftmost}}
The sites of a permutation $\pi \in \C_n$ are to the left of $\pi$, and are located below, above and between any two consecutive values $i$ and $i+1$ of $\pi$. The children of $\pi$ are the permutations of $\C_{n+1}$ obtained when inserting a leftmost element in an active site of~$\pi$.

Notice that each of the four insertion rules corresponds to adding a new element on one of the four sides of the square around the diagram that represent the permutation (see Figure \ref{fig:sites}).

\begin{figure}[ht]
\begin{center}
\scalebox{0.7}{\input{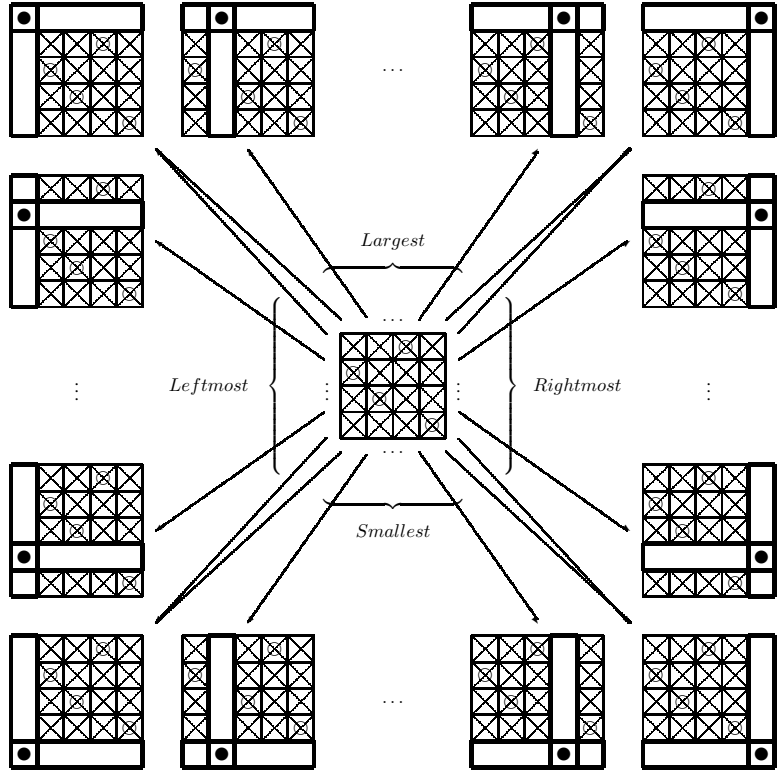}}
\end{center}
\caption{The sites corresponding to each insertion rule, on the diagram representation of permutation $\pi = 3241$, and the permutations resulting from the four insertion rules.} \label{fig:sites}
\end{figure}

\bigskip

\subsection{Rewriting system}

The shape of a generating tree associated to a set \C of permutations contains information, even without the permutations labeling the vertices, in particular for enumeration. Indeed, even when considering a generating tree where the permutations labeling the vertices have been erased, we still have a bijection between the vertices of this infinite tree and the permutations of \C, which maps the size of a permutation to the level a vertex in the tree (the level denoting the distance to the root $+1$ here).

{\em Rewriting systems} are a way to describe the shape of a generating tree without the need of labeling each vertex by a permutation. Instead, we can label the vertices of the tree by tuples -that are called \emph{labels}-, in such a way that the label of each vertex contains enough information on the corresponding permutation $\pi$ to build the labels of all the children of $\pi$. In general, these tuples indicate the values of some statistics on $\pi$, such as the size, the number of active sites, \ldots The shape of the tree is then completely described by a rewriting system on the labels, that encapsulate the parent-child relation on the permutations.

Such a rewriting system consists of a starting point $\ell_0$ (the label of permutation $1$) and a set of rewriting rules of the form $\ell \rightsquigarrow L$ with $L = \{ \ell_a, \ldots, \ell_p \}$, that describe the labels $\ell_a, \ldots, \ell_p$ of the children of a permutation whose label is $\ell$. Therefore, there is a bijection between permutations of size $n$ in \C and sequences of labels $(\ell_0, \ell_1, \ell_2 \ldots, \ell_{n-1})$ such that for any $i$, there is a rewriting rule $\ell_i \rightsquigarrow L_i$ in the system such that $\ell_{i+1} \in L_i$. 
Consequently, when the same rewriting system is obtained for two sets $\C$ and $\mathcal{D}$ of permutations, this implies that there is a bijection between $\C$ and $\mathcal{D}$, that preserves the size of the permutation. 
Notice that it is also possible to enrich the labels of the vertices to take into account the value of one or more statistics on the corresponding permutations. 
In the case there is a common rewriting system for $\C$ and $\mathcal{D}$ where the labels take into account statistics $(s_1, \ldots, s_k)$ in $\C$ and $(r_1, \ldots, r_k)$ in $\mathcal{D}$, then these statistics are equidistributed. Section \ref{sec:srs} and \ref{sec:sis_enum} provide examples of this use of generating trees.

\section{Permutations sorted by \srs}
\label{sec:srs}

\subsection{A simple rewriting system}

We prove that $\Id(\srs)$ is enumerated by the same sequence than $\Id(\sos)$, by describing a generating tree that is common for $\Id(\sos)$ and $\Id(\srs)$. This tree is in fact identical to the common generating tree for $Av(3 2 1 4, \bar{2} 4 1 3 5)$ and $Av(3 2 4 1, \bar{2} 4 1 5 3)$ given in \cite{DGG98}. Indeed, we have the following correspondence:

\begin{prop}
The $\mathbf{c} \circ \mathbf{i}$ operation provides a bijection between $\Id(\sos)$ (resp. $\Id(\srs)$) and $Av(3 2 1 4, \bar{2} 4 1 3 5)$ (resp. $Av(3 2 4 1, \bar{2} 4 1 5 3)$).
\label{prop:bijection_sos_av(...)}
\end{prop}

Because of the $\mathbf{c} \circ \mathbf{i}$ transformation, the insertion rule {\em Largest} that was used in the generating tree of \cite{DGG98} is naturally transformed into the insertion rule {\em Rightmost} in the generating tree for $\Id(\sos)$ and $\Id(\srs)$. This incremental construction of permutations where active sites are on the right is also known in the literature under the name of {\em staff representation} of permutations (see \cite{BFP05} for example).
Furthermore, for the same reason, the sites on the right for $\Id(\sos)$ and $\Id(\srs)$ are numbered from top to bottom, so that it mimics the numbering from left to right for $Av(3 2 1 4, \bar{2} 4 1 3 5)$ and $Av(3 2 4 1, \bar{2} 4 1 5 3)$.

\begin{thm}
The generating trees for both $\Id(\sos)$ and $\Id(\srs)$ with the insertion rule {\em Rightmost} are characterized by the following rewriting system
\begin{equation*}
\R_{\Phi} 
\begin{cases}
(2,1,(1)) & \\
(x,k,(p_1, \ldots p_k)) \quad \rightsquigarrow & (2+p_j, j, (p_1, \ldots, p_{j-1},i)) \quad \textrm{for } 1\leq j \leq k \textrm{ and } p_{j-1} < i \leq p_j \\
& (x+1, k+1, (p_1, \ldots p_k,i)) \quad \textrm{for } p_k < i \leq x
\end{cases}
\end{equation*}
where in the label $(x, k, (p_1, \ldots, p_{k}))$ of a permutation $\pi$, $x$ denotes the number of active sites of  $\pi$, $k$ is the number of right-to-left maxima in $\pi$, and $p_{\ell}$ denotes the number of active sites above the $\ell$-th right-to-left maximum in $\pi$ (for the decreasing order of their values).
\label{thm:system_conj6}
\end{thm}

In Theorem \ref{thm:system_conj6} and in its proof, we use the convention that $p_0 = 0$.

Figure \ref{fig:gen_tree} shows the first few levels of the generating tree corresponding to the rewriting system $\R_{\Phi}$.

\begin{figure}[ht]
\begin{center}
\scalebox{0.6}{\input{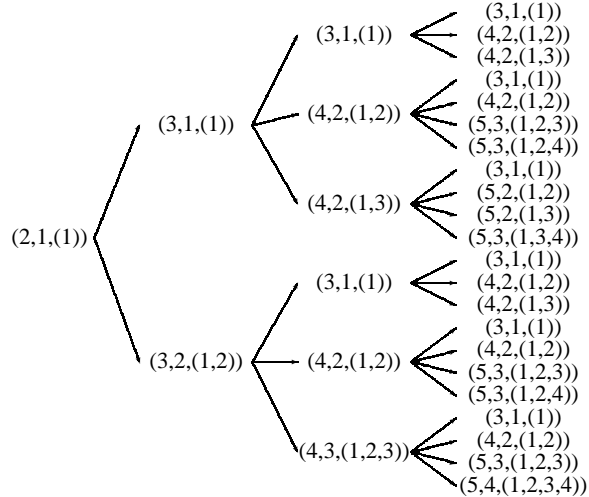}}
\end{center}
\vspace*{-0.7cm}
\caption{The first levels of the generating tree corresponding to the rewriting system $\R_{\Phi}$.} \label{fig:gen_tree}
\end{figure}

An immediate consequence of Theorem \ref{thm:system_conj6} is that the rewriting system $\R_{\Phi}$ provides a bijection $\Phi$ between $\Id(\sos)$ and $\Id(\srs)$ that preserves the size and the number of right-to-left maxima.

\begin{proof}
It is enough to use the bijective correspondence via the $\mathbf{c} \circ \mathbf{i}$ transformation between respectively $\Id(\sos)$ and $Av(3 2 1 4, \bar{2} 4 1 3 5)$, and $\Id(\srs)$ and $Av(3 2 4 1, \bar{2} 4 1 5 3)$. Indeed, Theorem \ref{thm:system_conj6} is a direct translation of Proposition 11 of \cite{DGG98} in this context. We however give the construction used in the proof, since it will have to be further analyzed in the next subsection. We omit the proof that this construction is correct, and refer the reader to \cite{DGG98} for details. 

%


Permutation $1$ belongs to $\Id(\sos)$ and $\Id(\srs)$, has two active sites, one right-to-left maximum, and one active site above this right-to-left maximum. 
Let us now examine the permutations that are obtained when inserting a rightmost element into the active site of a permutation $\pi$ of $\Id(\sos)$ (resp. $\Id(\srs)$) labeled $(x,k,(p_1, \ldots p_k))$. 

Consider $\pi \in \Id(\sos)$ (resp. $\Id(\srs)$), and one of its active site $s$, which is the $i$-th active site in the numbering from top to bottom. Denote by $\pi'$ the permutation obtained from $\pi$ by the insertion of a rightmost element in site $s$.

\begin{figure}[ht]
\begin{center}
\scalebox{0.5}{\input{insertion_2cases_SoS_SrS_large.jPicEdt}}
\end{center}
\vspace*{-0.7cm}
\caption{The two cases of insertion of a rightmost element into a permutation of $\Id(\sos)$ and $\Id(\srs)$.} \label{fig:insertion_conj6}
\end{figure}

Suppose first that site $s$ is above the rightmost element of $\pi$. Then define $j$ such that the largest right-to-left maximum of $\pi$ that is below $s$ is the $j$-th one, and $s_j$ to be the site that immediately below (resp. above) the $j$-th right-to-left maximum. Notice that in the case of $\Id(\sos)$, the site $s_j$ is always active in $\pi$. Then any site above $s_j$ (included) is active in $\pi'$ if and only if it was active in $\pi$, and the two sites that have been created around the inserted element are both active. As for the sites below $s_j$, they all become inactive (resp. all but the bottommost site, which is always active).

If on the contrary site $s$ is below the rightmost element of $\pi$, then any site is active in $\pi'$ if and only if it was active in $\pi$, and the two sites that have been created around the inserted element are both active. 

Figure \ref{fig:insertion_conj6} gives a graphical view of these two cases of insertion. In the first case, we have that $1 \leq j$, and $p_{j-1} < i  \leq p_j$, and the label of $\pi'$ is $(2+p_j, j, (p_1, \ldots, p_{j-1},i))$. In the second case, we have $p_k < i \leq x$ and the label of $\pi'$ is $(x+1, k+1, (p_1, \ldots p_k,i))$. 
\end{proof}

We can notice that in the above construction for $\Id(\sos)$ (resp. $\Id(\srs)$), the first, second and third (resp. last) sites, and the site above any right-to-left maximum are active (see \cite[Lemma 12, $(i)$ and $(ii)$]{DGG98} for the proof).

\subsection{Refinements of the rewriting system}

The rewriting system given by \cite{DGG98} for $Av(3 2 4 1, \bar{2} 4 1 5 3)$ and $Av(3 2 1 4, \bar{2} 4 1 3 5)$ is actually more precise, and takes into account the number of left-to-right maxima of a permutation $\pi$ but also the number of ascents in $\mathbf{i}(\pi)$. After the $\mathbf{c} \circ \mathbf{i}$ transformation, these statistics correspond to the number of right-to-left maxima and the number of descents for permutations in $\Id(\sos)$ and $\Id(\srs)$. Hence, these two statistics are preserved by the bijection $\Phi$ between $\Id(\sos)$ and $\Id(\srs)$. Actually every statistics $\mathrm{stat}$ of Theorem \ref{thm:conj6} is preserved by $\Phi$: for each of them, we can provide a refinement $\R_{\Phi}^{\mathrm{stat}}$ of the rewriting system $\R_{\Phi}$, thus completing the proof of Theorem \ref{thm:conj6}. Some of these rewriting systems are given below; the others are ommited for the sake of brevity but are obtained in a similar fashion.

\paragraph*{Number of left-to-right maxima}
The rewriting system $\R_{\Phi}$ can be refined as follows to account for the number of left-to-right maxima, in the part of the label denoted by $q$:
\begin{equation*}
\R_{\Phi}^{\lmax} 
\begin{cases}
(2,1,(1),1) & \\
(x,k,(p_1, \ldots p_k),q) \quad \rightsquigarrow & (2+p_1,1,(1),q+1)\\ 
& (2+p_j, j, (p_1, \ldots, p_{j-1},i),q) \quad \textrm{for } 1\leq j \leq k \\
& ~ \hspace*{4.45cm}~ \textrm{ and } p_{j-1} < i \leq p_j, i\neq 1 \\
& (x+1, k+1, (p_1, \ldots p_k,i),q) \quad \textrm{for } p_k < i \leq x
\end{cases}
\end{equation*}

\begin{proof}
The insertion into an active site does not change the number of left-to-right maxima of $\pi$, unless the insertion takes place in the topmost site of $\pi$. This site is always active, and in this case (corresponding to $i=j=1$), one left-to-right maximum is created.
\end{proof}

\paragraph*{Length of the leftmost decreasing run}
The rewriting system $\R_{\Phi}$ can be refined as follows to account for the length of the leftmost decreasing run, in the part of the label denoted by $q$:
\begin{equation*}
\R_{\Phi}^{\ldr} 
\begin{cases}
(2,1,(1),1,1) & \\
(x,k,(p_1, \ldots p_k),n,q) \quad \rightsquigarrow & (2+p_j, j, (p_1, \ldots, p_{j-1},i),n+1,q) \quad \textrm{for } 1\leq j \leq k \\
& ~ \hspace*{5.4cm} ~ \textrm{ and } p_{j-1} < i \leq p_j \\
& (x+1, k+1, (p_1, \ldots p_k,i),n+1,q) \quad \textrm{for } p_k < i < x \\
& (x+1, k+1, (p_1, \ldots p_k,x),n+1,q+\delta_{n,q})\\ 
&  ~ \hspace*{3.5cm} ~ \textrm{where } \delta_{n,q} \textrm{ is } 1 \textrm{ if } n=q \textrm{ and } 0 \textrm{ otherwise}
\end{cases}
\end{equation*}
It is necessary in this case to also account for more than just the value of the considered statistics in the label. Here, we furthermore introduced the size of the permutations, denoted by $n$.

\begin{proof}
The only case where the length of the leftmost decreasing run may change (being increased by $1$) is when $\pi$ is $n (n-1) \ldots 21$ and the insertion is performed in the bottommost site of $\pi$. For $\pi= n (n-1) \ldots 21$, the bottommost site is always active, for $\Id(\sos)$ and $\Id(\srs)$. An insertion in this special case is always an insertion in the $x$-th active site of $\pi$ and corresponds to the case where $q=n$.
\end{proof}

\paragraph*{The $\slmax$ statistics}

Recall that the $\slmax$ statistics is defined, for any permutation $\pi$, as the largest $k$ such that $\pi_1 \geq \pi_i$ for all $i \in [1..k]$. The rewriting system $\R_{\Phi}$ can be refined as follows to account for the $\slmax$ statistics in the part of the label denoted by $q$:
\begin{equation*}
\R_{\Phi}^{\slmax} 
\begin{cases}
(2,1,(1),1,1) & \\
(x,k,(p_1, \ldots p_k),q,n) \quad \rightsquigarrow & (2+p_1, 1, (1),q,n+1) \\ 
& (2+p_j, j, (p_1, \ldots, p_{j-1},i),q+\delta_{q,n},n+1) \quad \textrm{for } 1\leq j \leq k \\
& ~ \hspace*{4.9cm} ~ \textrm{ and } p_{j-1} < i \leq p_j, i\neq 1 \\
& (x+1, k+1, (p_1, \ldots p_k,i),q+\delta_{q,n},n+1) \quad \textrm{for } p_k < i \leq x\\
&  ~ \hspace*{3cm} ~ \textrm{where } \delta_{q,n} \textrm{ is } 1 \textrm{ if } q=n \textrm{ and } 0 \textrm{ otherwise}
\end{cases}
\end{equation*}
Here, we also introduced the size of the permutations, denoted by $n$.

\begin{proof}
The insertion in the topmost site of $\pi$ (which is always active) does not change the value of the $\slmax$ statistics. The insertion in any other site may change the value of this statistics, increasing it by $1$. This happens exactly when $\pi$ starts with its maximum, and this situation is characterized by the equality $q=n$.
\end{proof}

\section{Permutations sorted by \sis: enumerative result}
\label{sec:sis_enum}

We describe a common generating tree for $\Id(\sis)$ and $Av(2\dd 1 4 \dd 3, 2\dd 4 1 \dd 3)$. This is not the class of Baxter permutations, but they are equi-enumerated, as proved in~\cite{Gui95} or~\cite{Gir11}. This will prove Theorem~\ref{thm:conj7}.

\begin{thm}
The generating trees for both $\Id(\sis)$ and $Av(2\dd 1 4 \dd 3, 2\dd 4 1 \dd 3)$, with the insertion rules {\em Smallest} and {\em Largest} respectively, are characterized by the following rewriting system:
\begin{equation*}
\R_{\Psi} 
\begin{cases}
(2,0) & \\
(r,s) \quad \rightsquigarrow & (i+1,r+s-i) \textrm{ for } 1\leq i \leq r  \\
& (r,s-j) \textrm{ for } 1\leq j \leq s
\end{cases}
\end{equation*}
where the labels $(r,s)$ of a permutation $\pi$ are interpreted as follows:

For any permutation $\pi$ of size $n$ in $\Id(\sis)$:
\vspace{-6.5pt}
\begin{itemize}\setlength{\itemsep}{-0.8pt}
 \item $r$ is the index of the second element in the first ascent of $\pi$ (or $n+1$ if $\pi = n (n-1) \cdots 2 1$),
 \item $s$ is the number of active sites to the right of the first ascent of $\pi$.
\end{itemize}
\vspace{-6.5pt}

For any permutation $\pi$ of size $n$ in $Av(2\dd 1 4 \dd 3, 2\dd 4 1 \dd 3)$:
\vspace{-6.5pt}
\begin{itemize}\setlength{\itemsep}{-0.8pt}
 \item $r$ is the number of active sites to the left of $n$ plus $1$ (the site following $n$ immediately),
 \item $s$ is the number of active sites to the right of $n$ minus $1$ (the site following $n$ immediately).
\end{itemize}
\vspace{-9pt}
\label{thm:system_conj7}
\end{thm}

\begin{proof}
Let us first study $\Id(\sis) = Av(3412, 3 \dd 4 \dd 21)$ with the insertion rule {\em Smallest}. It is readily checked that if $\pi$ avoids $3412$ and $3 \dd 4 \dd 21$, then so does the permutation obtained from $\pi$ by deleting its smallest element and normalizing. This justifies that we have a generating tree, whose root $1$ is labeled $(2,0)$ according to the interpretation of the labels given in Theorem~\ref{thm:system_conj7}. It can further be noticed that a site which is inactive in $\pi$ cannot be active in any child of $\pi$.

It is easily proved that a site $x$ of $\pi$ is inactive if and only if there are elements $a<b<c$ such that $bca$ is a subsequence of $\pi$ and $x$ is either between $c$ and $a$ or follows $a$ immediately. In particular, every site that precedes the first ascent of $\pi$ and the site inside this first ascent are active. 

The number of such active sites correspond to $r$ in the label of $\pi$. The insertion of $1$ in such a site of $\pi$ (say the $i$-th such site) does not deactivate any site that was active in $\pi$, and the sites that were inactive in $\pi$ remain inactive. Hence, $r$ children of $\pi$ are obtained by insertions of this first type, whose labels are $(i+1, s+r-i)$ for $1\leq i \leq r$.

The other children $\pi'$ of $\pi$ are obtained by insertion of $1$ in active sites that are to the right of the first ascent of $\pi$, that we denote by $bc$. The insertion of $1$ in such a site (say the $j$-th such site) creates a subsequence $bc1$ in $\pi'$ so that the sites between $c$ and $1$ and just after $1$ are inactive in $\pi'$. So $j$ active sites to the right of the first ascent become inactive. For the other sites, as before they are active in $\pi'$ if and only if they were active in $\pi$. Hence, $s$ children of $\pi$ are obtained by insertions of this second type, whose labels are $(r, s-j)$ for $1\leq j \leq s$.

\smallskip

We now turn to the study of $Av(2\dd 1 4 \dd 3, 2\dd 4 1 \dd 3)$ with the insertion rule {\em Largest}. As above, if $\pi$ avoids $2\dd 1 4 \dd 3$ and $2\dd 4 1 \dd 3$, then so does the permutation obtained from $\pi$ by deleting its largest element, so that we have a generating tree. Its root $1$ is labeled $(2,0)$ according to the interpretation of the labels. Furthermore, denoting $n$ the size of $\pi$, the followings can be readily proved:
\vspace{-6.5pt}
\begin{itemize}\setlength{\itemsep}{-0.8pt}
 \item the site immediately to the right of $n$ is active;
 \item when inserting $n+1$ into an active site of $\pi$, the sites to the right of $n+1$ are active if and only if they were active in $\pi$;
 \item the sites to the left of $n$ are active if and only if they are located between two adjacent left-to-right maxima of $\pi$.
\end{itemize}
\vspace{-9pt}
With these three facts, a careful examination allows to prove that the insertion of $n+1$ in the $i$-th active site to the left of $n$, for $1\leq i \leq r-1$ (resp. in the site immediately to the right of $n$, resp. in the $j+1$-th active site to the right of $n$, for $1\leq j \leq s$), produces a child of $\pi$ labeled by $(i+1,r+s-i)$ (resp. $(r+1,s)$, resp. $(r,s-j)$).
\end{proof}

We claim that as in the case of Theorem~\ref{thm:conj6}, and although it requires to introduce additional statistics in the label, this rewriting system for $\Id(\sis)$ can be refined to follow the triple of statistics $(\des, \lmax, \comp)$. This however does not prove Conjecture~\ref{conj:conj7}, as we need first to examine how these statistics on Baxter permutations are transported on permutations avoiding $2\dd 1 4 \dd 3$ and $2\dd 4 1 \dd 3$, and second to refine the rewriting system for $Av(2\dd 1 4 \dd 3, 2\dd 4 1 \dd 3)$ according to these new statistics. The second part is likely to be solved when the statistics on $Av(2\dd 1 4 \dd 3, 2\dd 4 1 \dd 3)$ are known. However, they seem hard to characterize with the bijection of~\cite{Gui95}: because it relies on a generating tree, the one-to-one correspondence is not effective. On the contrary, \cite{Gir11} has recently given a \emph{constructive} bijection between $Av(2\dd 1 4 \dd 3, 2\dd 4 1 \dd 3)$ and the set of Baxter permutations. This bijection actually establishes a one-to-one correspondence from each of these to pairs of twin binary trees. Hence, a promising path to conclude the proof of Conjecture~\ref{conj:conj7} is to interpret the statistics $(\des, \lmax, \comp)$ on Baxter permutations as some triple of statistics on pairs of twin binary tree, and then to interpret these new statistics on the permutations of $Av(2\dd 1 4 \dd 3, 2\dd 4 1 \dd 3)$. This is indeed a work in progress.

\section*{Acknowledgements}

We thank Anders Claesson, Mark Dukes, and Einar Steingr\'imsson for bringing this problem to our attention, and for sharing their conjectures in \cite{ClDuSt}. We also thank Michael Albert, Mike Atkinson, Anders Claesson and Mark Dukes, with whom this research project started, during a visit of Mathilde Bouvel to Otago University.

\begin{small}
\bibliographystyle{abbrvnat}
\bibliography{biblio_final}

\begin{thebibliography}{15}
\providecommand{\natexlab}[1]{#1}
\providecommand{\url}[1]{\texttt{#1}}
\expandafter\ifx\csname urlstyle\endcsname\relax
  \providecommand{\doi}[1]{doi: #1}\else
  \providecommand{\doi}{doi: \begingroup \urlstyle{rm}\Url}\fi

\bibitem[Albert et~al.(2010)Albert, Atkinson, Bouvel, Claesson, and
  Dukes]{thm1.2}
M.~H. Albert, M.~D. Atkinson, M.~Bouvel, A.~Claesson, and M.~Dukes.
\newblock Private communication.
\newblock 2010.

\bibitem[Albert et~al.(2011)Albert, Atkinson, Bouvel, Claesson, and
  Dukes]{AlAtBoClDu11}
M.~H. Albert, M.~D. Atkinson, M.~Bouvel, A.~Claesson, and M.~Dukes.
\newblock On the inverse image of pattern classes under bubble sort.
\newblock \emph{Journal of Combinatorics}, 2\penalty0 (2), 2011.

\bibitem[Barcucci et~al.(1999)Barcucci, Del~Lungo, Pergola, and
  Pinzani]{BDPP99}
E.~Barcucci, A.~Del~Lungo, E.~Pergola, and R.~Pinzani.
\newblock {E}{C}{O}: A methodology for the enumeration of combinatorial
  objects.
\newblock \emph{J. Difference Equ. Appl.}, 5:\penalty0 435--490, 1999.

\bibitem[Baxter(1964)]{Bax64}
G.~Baxter.
\newblock On fixed points of the composite of commuting functions.
\newblock \emph{Proceedings of the American Mathematical Society}, 15:\penalty0
  851--855, 1964.

\bibitem[Bernini et~al.(2005)Bernini, Ferrari, and Pinzani]{BFP05}
A.~Bernini, L.~Ferrari, and R.~Pinzani.
\newblock Enumerating permutations avoiding three {B}abson-{S}teingr\'imsson
  patterns.
\newblock \emph{Annals of Combinatorics}, 9:\penalty0 137--162, 2005.

\bibitem[Chung et~al.(1978)Chung, Graham, Hoggatt, and Kleiman]{CGHK78}
F.~R.~K. Chung, R.~L. Graham, V.~E.~J. Hoggatt, and M.~Kleiman.
\newblock The number of {B}axter permutations.
\newblock \emph{J. Combin. Theory Ser. A}, 24\penalty0 (3):\penalty0 382--394,
  1978.

\bibitem[Claesson and Kitaev(2008)]{ClKi08}
A.~Claesson and S.~Kitaev.
\newblock Classification of bijections between 321- and 132-avoiding
  permutations.
\newblock \emph{S\'eminaire Lotharingien de Combinatoire}, 60, 2008.
\newblock Article B60d.

\bibitem[Claesson et~al.(2007)Claesson, Dukes, and Steingr\'imsson]{ClDuSt}
A.~Claesson, M.~Dukes, and E.~Steingr\'imsson.
\newblock Private communication.
\newblock 2007.

\bibitem[Dulucq et~al.(1998)Dulucq, Gire, and Guibert]{DGG98}
S.~Dulucq, S.~Gire, and O.~Guibert.
\newblock A combinatorial proof of {J}.~{W}est's conjecture.
\newblock \emph{Discrete Mathematics}, 187:\penalty0 71--96, 1998.

\bibitem[Giraudo(2011)]{Gir11}
S.~Giraudo.
\newblock Algebraic and combinatorial structures on {B}axter permutations.
\newblock In \emph{FPSAC 2011, DMTCS proc. AO}, pages 387--398, 2011.
\newblock Full version available at \url{http://arxiv.org/abs/1011.4288}.

\bibitem[Guibert(1995)]{Gui95}
O.~Guibert.
\newblock \emph{Combinatoire des permutations \`a motifs exclus en liaison avec
  mots, cartes planaires et tableaux de Young}.
\newblock PhD thesis, Universit\'e Bordeaux 1, 1995.

\bibitem[Knuth(1973)]{Knuth:ArtComputerProgramming:1:1973}
D.~E. Knuth.
\newblock \emph{Fundamental Algorithms}, volume~1 of \emph{The Art of Computer
  Programming}.
\newblock Addison-Wesley, Reading MA, 3rd edition, 1973.

\bibitem[Ouchterlony(2005)]{Ouc05}
E.~Ouchterlony.
\newblock \emph{On Young tableau involutions and patterns in permutations}.
\newblock PhD thesis, Link\"opings Universitet, 2005.

\bibitem[West(1993)]{West93}
J.~West.
\newblock Sorting twice through a stack.
\newblock \emph{Theoretical Computer Science}, 117\penalty0 (1-2):\penalty0
  303--313, 1993.

\bibitem[West(1995)]{Wes95}
J.~West.
\newblock Generating trees and the {C}atalan and {S}chr\"oder numbers.
\newblock \emph{Discrete Mathematics}, 146:\penalty0 247--262, 1995.

\end{thebibliography}
\end{small}

\end{document}